\documentclass{amsart}
\usepackage{amssymb,amsthm,amsmath,epsfig,latexsym,xypic,supertabular}
\usepackage{calc}
\usepackage{latexsym}
\usepackage{amscd,amssymb,subfigure,hyperref}
\usepackage[arrow,matrix,graph,frame,poly,arc,tips]{xy}
\usepackage{color}
\begin{document}

\newcommand{\mmbox}[1]{\mbox{${#1}$}}
\newcommand{\affine}[1]{\mmbox{{\mathbb A}^{#1}}}
\newcommand{\Ann}[1]{\mmbox{{\rm Ann}({#1})}}
\newcommand{\caps}[3]{\mmbox{{#1}_{#2} \cap \ldots \cap {#1}_{#3}}}
\newcommand{\N}{{\mathbb N}}
\newcommand{\Z}{{\mathbb Z}}
\newcommand{\Q}{{\mathbb Q}}
\newcommand{\R}{{\mathbb R}}
\newcommand{\KK}{{\mathbb K}}
\newcommand{\A}{{\mathcal A}}
\newcommand{\B}{{\mathcal B}}
\newcommand{\OO}{{\mathcal O}}
\newcommand{\C}{{\mathbb C}}
\newcommand{\PP}{{\mathbb P}}
\newcommand{\OS}{{T^d(X,p)}}

\newcommand{\Gor}{\mathop{\rm Gor}\nolimits}
\newcommand{\reg}{\mathop{\rm reg}\nolimits}
\newcommand{\charr}{\mathop{\rm char}\nolimits}
\newcommand{\ann}{\mathop{\rm ann}\nolimits}
\newcommand{\gin}{\mathop{\rm gin}\nolimits}
\newcommand{\Tor}{\mathop{\rm Tor}\nolimits}
\newcommand{\Ext}{\mathop{\rm Ext}\nolimits}
\newcommand{\Hom}{\mathop{\rm Hom}\nolimits}
\newcommand{\Sym}{\mathop{\rm Sym}\nolimits}
\newcommand{\im}{\mathop{\rm im}\nolimits}
\newcommand{\rk}{\mathop{\rm rk}\nolimits}
\newcommand{\codim}{\mathop{\rm codim}\nolimits}
\newcommand{\supp}{\mathop{\rm supp}\nolimits}
\newcommand{\coker}{\mathop{\rm coker}\nolimits}
\newcommand{\st}{\mathop{\rm st}\nolimits}
\newcommand{\lk}{\mathop{\rm lk}\nolimits}

\def\hess{\mathrm{hess}}
\def\Hess{\mathrm{Hess}}

\sloppy

\newtheorem{thm}{Theorem}[section]
\newtheorem*{thm*}{Theorem}
\newtheorem{defn}[thm]{Definition}
\newtheorem{prop}[thm]{Proposition}
\newtheorem{pref}[thm]{}
\newtheorem*{prop*}{Proposition}
\newtheorem{conj}[thm]{Conjecture}
\newtheorem{lem}[thm]{Lemma}
\newtheorem{rmk}[thm]{Remark}
\newtheorem{cor}[thm]{Corollary}
\newtheorem{notation}[thm]{Notation}
\newtheorem{exm}[thm]{Example}
\newtheorem{comp}[thm]{Computation}
\newtheorem{quest}[thm]{Question}
\newtheorem{prob}[thm]{Problem}

\newcommand{\msp}{\renewcommand{\arraystretch}{.5}}
\newcommand{\rsp}{\renewcommand{\arraystretch}{1}}

\newenvironment{lmatrix}{\renewcommand{\arraystretch}{.5}\small
  \begin{pmatrix}} {\end{pmatrix}\renewcommand{\arraystretch}{1}}
\newenvironment{llmatrix}{\renewcommand{\arraystretch}{.5}\scriptsize
  \begin{pmatrix}} {\end{pmatrix}\renewcommand{\arraystretch}{1}}
\newenvironment{larray}{\renewcommand{\arraystretch}{.5}\begin{array}}
  {\end{array}\renewcommand{\arraystretch}{1}}

  \newenvironment{changemargin}[2]{%
\begin{list}{}{%
\setlength{\topsep}{0pt}%
\setlength{\leftmargin}{#1}%
\setlength{\rightmargin}{#2}%
\setlength{\listparindent}{\parindent}%
\setlength{\itemindent}{\parindent}%
\setlength{\parsep}{\parskip}%
}%
\item[]}{\end{list}}

\title[AG algebras with Hilbert function $(1,4,k,k,4,1)$]
{A note on Artin Gorenstein algebras  with Hilbert function $(1,4,k,k,4,1)$}
\author[Nancy Abdallah]{Nancy Abdallah}
\address{Department of Engineering,
University of Bor\aa{}s,
Bor\aa{}s, Sweden}
\email{\href{mailto:nancy.abdallah@hb.se}{nancy.abdallah@hb.se}}
\urladdr{\href{https://nancyabdallah.info}%
{https://nancyabdallah.info}}


\begin{abstract}
\noindent We study the free resolutions of some Artin Gorenstein algebras of Hilbert function $(1,4,k,k,4,1)$ and we prove that all such algebras have the strong Lefschetz property if they have the weak Lefschetz property. In the case $k=4$ we prove that the Hilbert function alone fixes the Betti table. For higher $k$ stronger conditions on the algebras are needed to fix the Betti table. In particular, if the algebra is a complete intersection or if it is defined by an equigenerated ideal then the Betti table is unique. \end{abstract}
\vskip -.2in
\maketitle
\vskip -.2in
\renewcommand{\thethm}{\thesection.\arabic{thm}}
\setcounter{thm}{0}
\vskip -.2in
\section{Introduction}
 Let $S=\KK[x_1,\dots,x_n]$ be the standard graded polynomial ring over a field $\KK$ of characteristic zero, and $I$ a nondegenerate homogeneous ideal (containing no linear form) such that $A=S/I$ is an Artinian algebra. The Hilbert function of $A$ is the function $H(A,i):\Z\rightarrow\Z$, $H(A,i)=\dim_{\KK}A_i$. These dimensions are given by the $H$-vector of $A$, $H(A)=(1,h_1,h_2,\dots,h_s)$ where $h_i=H(A,i)>0$. The integer $s$ is called socle degree of $A$ (the highest integer $i$ for which $\dim_{\KK}A_i\ne 0$). 

Consider the ring $R=\KK[X_1,\dots, X_n]$. $S$ acts on $R$ by differentiation $$x_i(X_j)=\frac{\partial X_j}{\partial X_i}=\delta_{ij}.$$
Define $I_F=\ann_S(F)$ for a homogeneous polynomial $F$ of degree $s$. Macaulay shows in \cite{M}, that $S/I_F$ is Gorenstein of socle degree $s$, and that for every Artin Gorenstein (AG) algebra $A=S/I$ of socle degree $s$, $I=I_F$ for some homogenous polynomial $F\in R_s$. $F$ is called the Macaulay dual generator of $I$ and $I$ the Macaulay inverse system of $F$. 

For a finitely $\Z$-graded $S$-module $M$, the Hilbert Syzygy Theorem \cite{e} ensures the  existence of  a minimal graded finite free resolution, which is an exact sequence
\begin{equation*}
	0 \longrightarrow F_i \longrightarrow F_{i-1} \longrightarrow \cdots \longrightarrow F_0 \longrightarrow M \longrightarrow 0, 
\end{equation*}
for $i\le n$ and  $F_i \simeq \oplus_{j} S(-j)^{b_{i,j}}$ with $b_{i,j} \in \Z$. This information can be captured by the \textit{Betti table}, an array whose entry in position $(i,j)$ (reading over and down) is $b_{i,i+j}$ (see \cite{e}). With this indexing, the bottom row of the Betti table encodes the regularity of $M$. When $M=S/I$ is Artin Gorenstein, the regularity of $M$ coincides with its socle degree. 
\begin{exm}
	For $F=y_1y_2y_3 \in \KK[y_1,y_2,y_3]$, we have $I_F = \langle x_1^2,x_2^2,x_3^2 \rangle$ and the minimal free resolution is given by the Koszul complex
	\[
	0 \longrightarrow S(-6) \longrightarrow S(-4)^3 \longrightarrow S(-2)^3 \longrightarrow S \longrightarrow S/I_F \longrightarrow 0, 
	\]
	which in Betti table notation is written as
	\vskip -1.5in
		\begin{center}
		{\scriptsize \begin{verbatim}
				
				+--------------+
				|       0 1 2 3|
				|total: 1 3 3 1|
				|    0: 1 . . .|
				|    1: . 3 . .|
				|    2: . . 3 .|
				|    3: . . . 1|
				+--------------+
			\end{verbatim}
		}
	\end{center}
\end{exm}
For codimension 4 regularity (socle degree) 3 Artin Gorenstein (AG) algebras, the $H$-vector is $(1,4,4,1)$ and there are three possible Betti tables (see \cite{AS}). For AG algebras with $H$-vector $(1,4,k,4,1)$, $k=4,\dots,10$ (codimenion 4 regularity 4), there are 16 possible betti tables (see \cite{SSY}). In both cases the corresponding algebras have the weak Lefschetz property (see \cite{G}). In this note, we study the structure of AG algebras with regularity 5. In this case the $H$-vector is $(1,4,k,k,4,1)$, $k=4,\dots,10$. We prove first that for these algebras the weak Lefschetz property implies the strong Lefschetz property, which was noticed in \cite{AS} by looking at the Jordan type for a general linear form. In \cite{AS}, the authors conjectured also that there are 36 possible Betti tables for such AG algebras. Among these 36 Betti tables, there is only one that corresponds to the case $k=4$, and one that corresponds to a complete intersection. The main results of this paper are Theorems \ref{k4thm} and \ref{CIthm} where we prove that the conjecture holds for these two cases.

In Section 2, we recall some definitions and results about Lefschetz properties and we investigate these properties in the case of algebras with codimension 4 and regularity 5. AG algebras of $H$-vector $(1,4,4,4,4,1)$ are considered in Section 3. Section 4 is devoted to the study of complete intersections. We conclude in Section 5 with some open problems.

\section{Lefschetz Properties}

\begin{defn}
	Let $A$ be an Artinian $\Z$-graded algebra. We say that 
	\begin{enumerate}
		\item $A$ has the weak Lefschetz property (WLP) if there exists a linear form $\ell\in A_1$ such that the multiplication map $\ell: A_i\rightarrow A_{i+1}$ has maximal rank for all $i$. $\ell$ is called a weak Lefschetz element of $A$.
		\item $A$ has the strong Lefschetz property (SLP), if there is a linear form $\ell$ such that the multiplication map $\ell^k:A_i\rightarrow A_{i+k}$ has maximal rank for all $i$ and $k$.	$\ell$ is called a strong Lefschetz element $A$.
	\end{enumerate}
\end{defn}

When $A$ is Gorenstein, $A$ has SLP if there is a linear form $\ell$ such that the multiplication map $\ell^{s-2i}:A_i\rightarrow A_{s-i}$ has maximal rank (bijective) for all $i=1,\dots,\lfloor\frac{s}{2}\rfloor$, where $s$ is the socle degree of $A$. 

It is clear from the definition that if $A$ has SLP then $A$ has WLP. In codimension 2, all standard graded Artinian algebras have SLP provided that the field has characteristic zero or $p>s$ (\cite{HMNW}). In codimension 3, all complete intersections (CI) have WLP (\cite{HMNW}), and in \cite{BMMNZ2}, the authors conjectured that all codimension 3 AG algebras have WLP. In a recent paper (\cite{AAISY}), we prove that all codimension 3 AG algebras with Sperner number (maximal value of $H$-vector) at most 6, as well as algebras of $H$-vector $(1,3,4,5,\dots,(s-1),s^k,(s-1),\dots,3,1)$, have SLP. Our computation in \cite{AS} shows that for codimension 4 AG algebras with socle degree 5, the SLP holds when the Sperner number is at most 6, and the bound is sharp. 

We recall now the definition of Hessians and its relation with strong Lefschetz properties. 
\begin{defn}(\cite{MW}, Definition 3.1)
	Let $F$ be a polynomial in $R$ and $A=S/\ann_S F$ be its associated AG algebra. Let $\mathcal{B}_i=\{\alpha_m^{(i)} \}_m$ be a $\KK$-basis of $A_i$. The $i$-th Hessian matrix of $F$ with respect to $\mathcal{B}_i$ is defined by
	$$(\Hess^i(F))_{u,v}=(\alpha_u^{(i)}\alpha_v^{(i)}\circ F).$$
	We write $\hess^i(F)$ for the determinant of $\Hess^i(F)$ which is independent of the basis $\mathcal{B}_i$ up to a non-zero constant multiple.  For $i=1$, $\Hess^1(F)$ coincides with the usual Hessian. 
\end{defn}
The following result due to T. Maeno and J. Watanabe gives a strong connection between the Hessian matrices and the Strong Lefschetz propeties. 

\begin{thm}\label{hesscrit}
	Let $A=S/\ann_S F$ be an AG algebra. $\ell=a_1x_1+\cdots+a_nx_n\in A_1$ is a strong Lefschetz element of $A$ if and only if $\hess^i(F)(a_1,\dots,a_n)\neq 0$ for all $i=1,\dots,\lfloor\frac{s}{2}\rfloor$. In particular, the multiplication map $\ell^{s-2i}:A_i\rightarrow A_{s-i}$ is bijective if and only if  $\hess^i(F)(a_1,\dots,a_n)\neq 0$.
\end{thm}
The following theorem due to P.~Gordan and M.~Noether in \cite{GoNo} was reproved in \cite[Theorem 1.2]{Los}, \cite{WatdeB}, \cite{BFP}.

\begin{thm}[Gordan-Noether]\label{GNlem} If $F$ is a form expressed in at most four variables and the characteristic of $\KK$ is zero, then the Hessian determinant $\hess(F)$ is identically zero  if and only if $F$ is  a cone: that is, if $F$ is annihilated by a linear form (element of $S_1$) in the Macaulay duality.
\end{thm}

In \cite{AS}, our computation showed that for AG algebras with $H$-vector $(1,4,k,k,4,1)$, $k=4,\dots,10$, whenever WLP holds, SLP holds as well. The following proposition proves that the claim holds.
\begin{prop}\label{WLPSLP}
	Let $A=S/I$ be an AG algebra where $I$ is a nondegenerate homogeneous ideal and $H(A)=(1,4,k,k,4,1)$, $k=4,\dots, 10$. If $A$ has WLP then $A$ has SLP.	
\end{prop}

\begin{proof}
	The first Hessian is non zero by Gordan Noether (Theorem \ref{GNlem}) and the second Hessian is non zero by the WLP property (the middle map is bijective). The result follows by Theorem \ref*{hesscrit}.
\end{proof}

\section{Artin Gorenstein algebras of $H$-vector $(1,4,4,4,4,1)$}

We consider in this section $AG$ algebras with $H$-vector $(1,4,4,4,4,1)$. In \cite{AS}, we conjectured that there is only one Betti table for this $H$-vector and that all these AG algebras have SLP. We prove in this section that both claims hold. 

\begin{prop}
	Every AG algebra with $H$-vector $(1,4,4,4,4,1)$ has SLP. 
\end{prop}

\begin{proof}
	The fact that such an algebra has WLP follows from in \cite[Proposition 3.3]{AADFIMMMN}. By Proposition \ref{WLPSLP}, SLP holds as well. 
\end{proof}

\noindent We review the theorems of Macaulay and Gotzmann (\S 7.2 of \cite{cag}). For a graded algebra $S/I$ with Hilbert function $h_i$, write 
\[
h_i = \binom{a_i}{i} + \binom{a_{i-1}}{i-1}+ \cdots \mbox{ and }h_i^{\langle i \rangle} = \binom{a_i+1}{i+1} + \binom{a_{i-1}+1}{i}+ \cdots, 
\]
\noindent where $ a_i > a_{i-1} > \cdots$. Then we have
\begin{thm}[Macaulay]\label{Macaulay}
	In the setting above, 
	\[
	h_{i+1} \le h_i^{\langle i \rangle}.
	\]
\end{thm}
\begin{thm}[Gotzmann]\label{Gotzman} If $I$ is generated in a single degree $t$ and equality holds in Macaulay's formula in the first degree $t$, then 
	\[
	h_{t+j} = \binom{a_t+j}{t+j} + \binom{a_{t-1}+j-1}{t+j-1}+ \cdots
	\]
\end{thm}
\newpage 
We now prove our first main result.
\begin{thm}\label{k4thm}
	An AG algebra $A$ with Hilbert function $(1,4,4,4,4,1)$ has Betti table
	
	\begin{center}
		{\scriptsize \begin{verbatim}
				
				+-----------------+
				|       0 1  2 3 4|
				|total: 1 9 16 9 1|
				|    0: 1 .  . . .|
				|    1: . 6  8 3 .|
				|    2: . .  . . .|
				|    3: . .  . . .|
				|    4: . 3  8 6 .|
				|    5: . .  . . 1|
				+-----------------+
			\end{verbatim}
		}
	\end{center}
\end{thm}

\begin{proof}
	We know that the Betti table of a Gorenstein algebra is symmetric, and since $h_2(A)=4$, the ideal $I$ has 6 quadrics, therefore $b_{12}:=b_{1,2}=6=b_{37}$. To avoid many indices we write the Betti table as below (we ommit the row of total dimensions for simplicity).
		\begin{center}
		{\scriptsize \begin{verbatim}
				
				+--------------+
				|    0 1  2 3 4|
				| 0: 1 .  . . .|
				| 1: . 6  b c .|
				| 2: . d  e f .|
				| 3: . f  e d .|
				| 4: . c  b 6 .|
				| 5: . .  . . 1|
				+--------------+
			\end{verbatim}
		}
	\end{center}
Denote by $I_j$ the degree $j$ component of the ideal $I$. The Hilbert number $h_j(A)$ is given by $h_j(A)=\binom{3+j}{j}-\dim_{\KK}(I_j)$. For $i=3$ and $i=4$ we have,
$h_3(A)=20-6\cdot 4-d+b=4$ and $h_4(A)=35-6\cdot 10-4d-f+4b+e-c=4$. This gives the relations \begin{equation}\label{bdrel}
	b=d+8
\end{equation}
 and \begin{equation}\label{cefrel}
 	c-e+f=3
 \end{equation}
Consider now the algebra $A=S/J_2$ where $J_2$ is the ideal generated by the quadrics of $I$. Then $h_i(S/J_2)=h_i(A)$ for $i\le 2$ and the first two rows of the Betti tables of $A$ and $S/J_2$ coincide. Since $h_2(S/J_2)=4=\binom{3}{2}+\binom{1}{1}$, by Theorem \ref{Macaulay},  $h_3(S/J_2)\le h_2^{<2>}(S/J_2)=\binom{4}{3}+\binom{2}{2}=5$. On the other hand, $h_3(S/J_2)=20-6\cdot 4+b\le 5$ and $b\le 9$. Together with Equation 1, we get $b\in \{8,9\}$. We want to prove that the case $b=9$ cannot occur. 

Suppose $b=9$, i.e. $h_3(S/J_2)=h_2^{<2>}(S/J_2)$. By Theorem \ref{Gotzman}, $h_4(S/J_2)=\binom{3+2}{2+2}+\binom{1+2-1}{2+2-1}=5$, but $h_4(S/J_2)=35-6\cdot 10+9\cdot4+b_{24}(S/J_2)-c$. Therefore $c= b_{24}(S/J_2)+6\ge 6$.

We study now the Betti tables of $S/J_3$ and $S/J_4$ where $J_3$ and $J_4$ are the ideals generated by the generators of $I$ of degree at most 3 (respectively 4), and we prove that $c\in\{2,3\}$ which contradicts the inequality $c\ge 6$ in the case $b=9$. 

For $j=3,4$ the first $j$ rows of the Betti table of $S/J_j$ coincide with those of $A$, and $h_i(S/J_j)=h_i(A)$ for $i\leq j$. Then, $h_3(S/J_3)=4=\binom{4}{3}$ and by Macaulay $$h_3^{<3>}=5\ge h_4(S/J_3)=35-60-4d+4b+e-c=7+e-c.$$ Therefore, $c\ge 2+e\ge 2$. By Equation \ref{cefrel}, we get $f\leq 1$. On the other hand, 
$h_4(S/J_4)=4=\binom{4}{4}+\binom{3}{3}+\binom{2}{2}+\binom{1}{1}$, and again by Macaulay, $$h_4^{<4>}(S/J_4) =4\ge h_5(S/J_4)=56-120-10d-4f+10b+4e+e-4c-f.$$ Using Equations \ref{bdrel} and \ref{cefrel}, we get $c\leq 3$, a contradiction. Hence, $b=8$ and $d=0$.  

We have now, $b=8$, $c\in\{2,3\}$, $d=0$, $e=f+c-3$ and $f\in\{0,1\}$. It is enough to prove that $f=1$ cannot occur. 

Since $d=0$, the ideal $I$ has no degree three generators. This implies that $S/J_2$ and $A$ share the first three rows of the Betti table, and that $h_4(S/J_2)=35-60+32+e-c$. If $f=1$, $e-c=-2$, and $h_4(S/J_2)=5=\binom{5}{4}$, then $$h_4^{<4>}(S/J_2)=6\ge h_5(S/J_2)=56-120+80+4e+b_{25}(S/J_2)-4c-1,$$ which gives $b_{25}(S/J_2)\le -1$, a contradiction. The remaining case is the desired case which occurs by a computation using Macaulay2. 
\end{proof}
\begin{exm}
	The ideal $I=(x_1x_3-x_2x_4,x_2^2,x_2x_3,x_3^2,x_3x_4,x_4^2,x_1^4x_2,x_1^4x_4,x_1^5)$ defines an AG algebra of $H$-vector $(1,4,4,4,4,1)$ and Betti table as in Theorem \ref{k4thm}.
\end{exm}
\section{Complete intersection}
In the previous section, we saw that certain Hilbert functions lead to strong constraints on the possible Betti table for an AG ring. Hence it is natural to ask if there are other conditions which strongly constrain the Betti table, and in this section, we consider the case where $I$ is a CI. We prove that given codimension $c=4$ and the regularity $r=5$ of an AG algebra, the complete intersection condition fixes the Hilbert function and the Betti table.

To prove our theorem, we prove first that, for the same codimension and socle degree, there is exactly one Betti table for $A=S/I$ where $I$ is an equigenerated ideal. 

\begin{prop}\label{equig}
	Let $A=S/I$ be an AG algebra with socle degree 5. If $I$ is equigenerated then $H(A)=(1,4,10,10,4,1)$ and the Betti table is given by 
	
		\begin{center}
		{\scriptsize \begin{verbatim}
				
				+-------------------+
				|       0  1  2  3 4|
				|total: 1 10 18 10 1|
				|    0: 1  .  .  . .|
				|    1: .  .  .  . .|
				|    2: . 10  9  . .|
				|    3: .  .  9 10 .|
				|    4: .  .  .  . .|
				|    5: .  .  .  . 1|
				+-------------------+
			\end{verbatim}
		}
	\end{center}
\end{prop}

\begin{proof}
	If $I$ is equigenerated, the Betti table has one of the two forms 
		\begin{center}
		{\scriptsize \begin{verbatim}
				
    +-------------+    			      +-------------+
    |    0 1 2 3 4|	    		     |    0 1 2 3 4|
    | 0: 1 . . . .|		    	     | 0: 1 . . . .|
    | 1: . a b . .|			          | 1: . . . . .|
(1)	| 2: . . e . .|		    	(2)  | 2: . d e . .|
    | 3: . . e . .|		    	     | 3: . . e d .|
    | 4: . . b a .|			          | 4: . . . . .|
    | 5: . . . . 1|			          | 5: . . . . 1|
    +-------------+			          +-------------+
			\end{verbatim}
		}
	\end{center}
We prove first that Table (1) cannot occur. Suppose $A$ has Betti table (1). Since the $H$-vector of $A$ is $(1,4,k,k,4,1)$, $k=4,\dots, 10$, then $I$ has $a=10-k$ quadratic generators.We compute $h_i(A)$ for $i=3,4,5$. $h_i(A)=\binom{3+i}{i} - \dim_{\KK}I_i$ where $I_i$ is the $i^{th}$ graded component of the graded ideal $I$. 
For $i=3$, $h_3(A)=20-(4a-b)=k$; for $i=4$, $h_4(A)=35-(10a-4b-e)=4$; and for $i=5$, $h_5(A)=56-(20a-10b-4e-e)=1$. Therefore, $k,a,b,e$ are nonnegative integer solutions to the system 
\[ \begin{cases}
 	k&=10-a\\
 	k&=20-4a+b\\
 	4&=35-10a+4b+e\\
 	1&=56-20a+10b+5e
 \end{cases}\]

The only integer solution of this system is $(k,a,b,e)=(6,4,2,1)$ which gives a complete intersection, but the Betti table of a CI with 4 quadrics is 
	\begin{center}
	{\scriptsize \begin{verbatim}
			
			+----------------+
			|       0 1 2 3 4|
			|total: 1 4 6 4 1|
			|    0: 1 . . . .|
			|    1: . 4 . . .|
			|    2: . . 6 . .|
			|    3: . . . 4 .|
			|    4: . . . . 1|
			+----------------+
		\end{verbatim}
	}
\end{center}
and the socle degree in this case is 4, a contradiction. 

We know now that the Betti table of $A$ is of the form of Table (2). Since $I$ has no quadrics, $H(A)=(1,4,10,10,4,1)$ and therefore $I$ is generated by 10 cubics, so $d=10$. Now since $h_4(A)=35-\dim I_4=35-(4d-e)=4$ then $e=9$ and we get the desired Betti table. 
\end{proof}

We can now prove our second main theorem.
\begin{thm}\label{CIthm}
	Let $A$ be a complete intersection with $H$-vector $(1,4,k,k,4,1)$. Then the Betti table of $A$ is 
		
	\begin{center}
		{\scriptsize \begin{verbatim}
				
			   	+----------------+
			   	|       0 1 2 3 4|
		   		|total: 1 4 6 4 1|
	   			|    0: 1 . . . .|
	(3)			|    1: . 3 . . .|
			   	|    2: . 1 3 . .|
			   	|    3: . . 3 1 .|
			   	|    4: . . . 3 .|
			   	|    5: . . . . 1|
			   	+----------------+
			\end{verbatim}
		}
	\end{center}
In particular, the $H$-vector is $(1,4,7,7,4,1)$.	
\end{thm}

\begin{proof}
	Since $A$ is a CI, $I$ has at most four quadratic generators. Then $k$ cannot be 4 or 5. If $k=6$, then $I$ is an equigenerated ideal generated by $4$ quadrics. This is not possible by Proposition \ref{equig}. If $k=10$, then $I$ has ten cubic generators and is not a CI. We are left with the cases $k=7,8,9$. 
	
 	Suppose $k=9$. The ideal $I$ is generated by one quadric ($b_{12}=1$), and three other generators of degree 3, 4 or 5. Since it is not possible to have syzygies on one generator, $b_{23}=b_{34}=0$.  Therefore $h_3(A)=\binom{6}{3}-\dim I_3=20-(4\cdot 1+b_{13})=9$, and $b_{13}=7$ which does not give a CI. 
 	
 	Suppose now that $k=8$. In this case $I$ has two quadratic generators and two generators of degree 3, 4 or 5 (i.e., $b_{13}+b_{14}+b_{15}=2$). There is at most one linear syzygy on two quadrics. Then, $b_{23}$ is either 0 or 1. Therefore, $h_3(A)=20-(4\cdot 2+b_{13}-b_{23})=8$, and $b_{13}=4-b_{23}$ is either 3 or 4; a contradiction since $b_{13}\leq 2$.
 		
 	Hence, we proved that $H(A)=(1,4,7,7,4,1)$, so $b_{12}=3$ and $b_{13}+b_{14}+b_{15}=1$. If $b_{13}=0$, $h_3(A)=20-(4\cdot 3-b_{23})=7$ and $b_{23} =-1$, a contradiction. Therefore, the Betti table of $A$ is of the form \\		
 	\begin{center}
 		{\scriptsize \begin{verbatim}
 				
 				+--------------+
 				|    0 1  2 3 4|
 				| 0: 1 .  . . .|
 				| 1: . 3  b . .|
 				| 2: . 1  e . .|
 				| 3: . .  e 1 .|
 				| 4: . .  b 3 .|
 				| 5: . .  . . 1|
 				+--------------+
 			\end{verbatim}
 		} 
 	\end{center}
 Now $h_3(A)= 20-(12+1-b)=7$ and $b=0$; $h_4(A)=35-(3\cdot 10+1\cdot 4-e)=4$ and $e=3$. We get the desired table which is realizable (see the following example). 	
 \end{proof}

\begin{exm}\label{CImon}
	Let $I=(x_1^2,x_2^2,x_3^2,x_4^3)$. $A=S/I$ is a CI of $H$-vector $H(A)=(1,4,7,7,4,1)$ and Betti table as in Table (3). The Macaulay dual generator is $F=X_1X_2X_3X_4^2$.
\end{exm}
\begin{rmk}
	It was pointed out by Anthony Iarrobino that a part of the proof can be simplified using \cite[Theorem 2.3.11]{BH}: for a CI in codimension 4, if $D=(d_1,d_2,d_3,d_4)$ are the degrees generators of $I$ with $2\le d_1\le d_2 \le d_3 \le d_4$, then $s+4=\Sigma_i d_i$. For the case $(1,4,k,k,4,1)$, we have then only one choice for $D$: $D=(2,2,2,3)$. This forces $k$ to be 7 and gives the second and fourth column of the Betti table. 
\end{rmk}
\begin{rmk}
	Note that for the $H$-vector $(1,4,7,7,4,1)$ it is conjectured in \cite{AS} that there are 9 possible Betti tables, and that SLP holds for all AG algebras with the Betti tables below. 
			\begin{center}
		{\scriptsize \begin{verbatim}
				
				+-----------------+-----------------+-----------------+-----------------+
				|       0 1  2 3 4|       0 1  2 3 4|       0 1  2 3 4|       0 1  2 3 4|
				|total: 1 4  6 4 1|total: 1 6 10 6 1|total: 1 6 10 6 1|total: 1 7 12 7 1|
				|    0: 1 .  . . .|    0: 1 .  . . .|    0: 1 .  . . .|    0: 1 .  . . .|
				|    1: . 3  . . .|    1: . 3  2 . .|    1: . 3  2 . .|    1: . 3  . . .|
				|    2: . 1  3 . .|    2: . 3  3 . .|    2: . 2  4 1 .|    2: . 1  6 3 .|
				|    3: . .  3 1 .|    3: . .  3 3 .|    3: . 1  4 2 .|    3: . 3  6 1 .|
				|    4: . .  . 3 .|    4: . .  2 3 .|    4: . .  . 3 .|    4: . .  . 3 .|
				|    5: . .  . . 1|    5: . .  . . 1|    5: . .  . . 1|    5: . .  . . 1|
				+-----------------+-----------------+-----------------+-----------------+
				|       0 1  2 3 4|       0 1  2 3 4|       0 1  2 3 4|       0 1  2 3 4|
				|total: 1 7 12 7 1|total: 1 9 16 9 1|total: 1 9 16 9 1|total: 1 6 10 6 1|
				|    0: 1 .  . . .|    0: 1 .  . . .|    0: 1 .  . . .|    0: 1 .  . . .|
				|    1: . 3  2 . .|    1: . 3  2 . .|    1: . 3  3 1 .|    1: . 3  3 1 .|
				|    2: . 3  4 . .|    2: . 3  6 3 .|    2: . 4  5 1 .|    2: . 4  7 3 .|
				|    3: . 1  4 1 .|    3: . 3  6 3 .|    3: . 1  5 4 .|    3: . 3  7 4 .|
				|    4: . .  2 3 .|    4: . .  2 3 .|    4: . 1  3 3 .|    4: . 1  3 3 .|
				|    5: . .  . . 1|    5: . .  . . 1|    5: . .  . . 1|    5: . .  . . 1|
				+-----------------+-----------------+-----------------+-----------------+
			\end{verbatim}
		}
	\end{center}
For AG algebras with the table 
		\begin{center}
	{\scriptsize \begin{verbatim}
			
			+-----------------+
			|       0 1  2 3 4|
			|total: 1 8 14 8 1|
			|    0: 1 .  . . .|
			|    1: . 3  1 . .|
			|    2: . 2  6 3 .|
			|    3: . 3  6 2 .|
			|    4: . .  1 3 .|
			|    5: . .  . . 1|
			+-----------------+
		\end{verbatim}
	}
\end{center}
there are examples where SLP holds and other examples where WLP fails. 
\end{rmk}

\section{Open problems}
We conclude this note with some open problems. Let $\Gor(H)$ be the set of all AG with $H$-vector $H$. In \cite{Di} S.J. Diesel showed that  in three variables and over an algebraically closed field the variety $\Gor(H)$ is irreducible.  In four variables, Boij showed in \cite{B} that $\Gor(H)$ can consist of several components. 
\begin{prob}
It would be interesting to investigate the loci in $\Gor(H)$ corresponding to the possible betti tables. For example, if we order Betti tables as a poset with $\beta\ge \beta'$ if $b_{i,j}\ge b'_{i,j}$ for all $i,j$, then for the $H$-vector $(1,4,5,5,4,1)$ there is a strict inclusion of the Betti tables whereas for the $(1,4,6,6,4,1)$ this is not the case. 
\end{prob}

The  Jordan  type  of  a  linear  form  $\ell$ acting  on  an  Artinian  algebra $A$ is the partition describing the Jordan block decomposition of the matrix representing the multiplication by a linear form between the graded pieces of $A_i\stackrel{\cdot \ell}{\longrightarrow}A_{i+1}$.  In \cite{HMMNWW} the authors show that $A$ has SLP if and only if the partition $P_{A,\ell}$ is conjugate to the $H$-vector. In \cite{AS}, by a computation with Macaulay2 we see that the Jordan type of a CI with Betti table as in (3) is $(6,4,4,4,2,2,2)$ for a general linear form, which means that such algebras have SLP. 
\begin{prob}
Show that all codimension 4 CI with socle degree 5 have SLP. 
\end{prob}
\begin{exm}
	The algebra  $A=S/(x_1^2,x_2^2,x_3x_4,x_3^3-x_4^3)$ is a CI. For $\ell=x_1+x_2+x_3+x_4$ the Jordan type is $(6,4,4,4,2,2,2)$ which means that $A$ has SLP. 
\end{exm}

\vskip .05in
\noindent{\bf Acknowledgements} The author would like to thank the referee, Anthony Iarrobino and Hal Schenck for their useful remarks, and Anthony Iarrobino, Pedro Macias Marquez, Maria Evelina Rossi and Jean Vallès for organizing the conference AMS-SMF-EMS special session on Deformation of Artinian algebras and Jordan Type in Grenoble.

\end{document}